\documentclass[10pt]{amsart}
\usepackage{amssymb}
\usepackage{color}
\usepackage{graphicx}
\usepackage{float}
\usepackage[all,cmtip]{xy}
\newcommand{\FX}{{\mathfrak{X}}}
\newcommand{\SF}{{\mathcal{F}}}
\newcommand{\SL}{{\mathcal{L}}}
\newcommand{\BW}{{\textbf{W}}}
\newcommand{\BN}{{\textbf{N}}}

\newcommand{\BK}{{\textbf{K}}}

\newcommand{\im}{\operatorname{Image}}

\newcommand{\lra}{\longrightarrow}

\newtheorem{proposition}{Proposition}
\newtheorem{theorem}{Theorem}
\newtheorem{definition}{Definition}
\newtheorem{lemma}{Lemma}

\newtheorem{corollary}{Corollary}
\newtheorem{remark}{Remark}

\begin{document}

\title{Foliated vector fields without periodic orbits}

\date{March, 2015}


\author{Daniel Peralta--Salas}
\address{Instituto de Ciencias Matem\'aticas -- CSIC.
C. Nicol\'as Cabrera, 13--15, 28049, Madrid, Spain.}
\email{dperalta@icmat.es}

\author{\'Alvaro del Pino}
\address{Universidad Aut\'onoma de Madrid and Instituto de Ciencias Matem\'aticas -- CSIC.
C. Nicol\'as Cabrera, 13--15, 28049, Madrid, Spain.}
\email{alvaro.delpino@icmat.es}

\author{Francisco Presas}
\address{Instituto de Ciencias Matem\'aticas -- CSIC.
C. Nicol\'as Cabrera, 13--15, 28049, Madrid, Spain.}
\email{fpresas@icmat.es}

\begin{abstract}
In this article parametric versions of Wilson's plug and Kuperberg's plug are discussed. We show that there is a weak homotopy equivalence induced by the inclusion between the space of non--singular vector fields tangent to a foliation and its subspace comprised of those without closed orbits, as long as the leaves of the foliation have dimension at least 3. We contrast this with the case of foliations by surfaces in 3--manifolds.
\end{abstract}

\maketitle
\section{Introduction}
The Seifert Conjecture \cite{Sei} stated that all non--singular vector fields in $\mathbb{S}^3$ have at least one closed orbit. A construction by Wilson, \cite{Wil}, shows that any non--singular vector field in a 3--dimensional manifold can be homotoped through non--singular vector fields to a vector field with only finitely many closed orbits, both of them with the same degree of differentiability. The same construction also proves that, for manifolds of dimension at least 4, any non--singular vector field can be homotoped through non--singular vector fields to one without closed orbits. 

After a result of Schweitzer \cite{Sch} proving that the Seifert Conjecture does not hold under $C^1$ regularity, Krystyna Kuperberg settled Seifert's Conjecture in the negative \cite{Kup}, by showing that any smooth non--singular vector field in a 3--dimensional manifold can be smoothly homotoped to a vector field with no closed orbits. 

These results can be restated as follows. Given a manifold $M$, denote by $\FX_{ns}(M)$ the space of smooth non--singular vector fields on $M$ and by $\FX_{no}(M)$ the space of smooth non--singular vector fields with no closed orbits, both of them endowed with the $C^\infty$--topology. Then Wilson's and Kuperberg's constructions show that the inclusion
\[ \iota_n: \FX_{no}(M) \lra \FX_{ns}(M) \]
induces a surjection in $\pi_0$ as long as $\dim(M) \geq 3$. 

Both Wilson's and Kuperberg's constructions are based around the notion of a plug. A plug is a local model for modifying a vector field in a flowbox. Wilson's plug in higher dimensions traps a non--empty open subset of orbits, while it creates no new closed ones. Kuperberg's plug in dimension $3$ creates no new closed orbits, and a later result by Matsumoto, \cite{Mat}, shows that the set of orbits that are trapped in Kuperberg's plug contains a non--empty open subset. 

Let $(M^{n+m}, \SF^n)$, $n \geq 3$, be a closed smooth ($n+m$)-dimensional manifold endowed with a smooth foliation of codimension $m$. Denote by $\FX_{ns}(M, \SF)$ and $\FX_{no}(M, \SF)$ the subsets of, respectively, $\FX_{ns}(M)$ and $\FX_{no}(M)$ consisting of vector fields tangent to $\SF$. The main result of this note is the following:

\begin{theorem} \label{thm:main}
The inclusion:
\[ \iota_n: \FX_{no}(M, \SF) \lra \FX_{ns}(M, \SF) \]
is a weak homotopy equivalence. That is, the induced maps in homotopy
\[ \pi_k\iota_n: \pi_k\FX_{no}(M, \SF) \lra \pi_k\FX_{ns}(M, \SF) \]
are isomorphisms. 
\end{theorem}

In particular, in the case where $\SF$ is comprised of a single leaf, the whole of $M$, this recovers and improves the results of Wilson and Kuperberg, that dealt only with $\pi_0$. Also, it shows that the \textsl{foliated Seifert conjecture} -- every vector field tangent to a foliation has a closed orbit -- does not hold for foliations of dimension $n\geq 3$. 

\begin{remark}
For the result not to be trivial, the space $\FX_{ns}(M, \SF)$ should be non--empty. This reduces to a purely algebraic topology question. For instance, in the classical case (foliation consisting of a single leaf) the necessary and sufficient condition for non--triviality is $\chi(M)=0$. Another example: if you assume that the manifold is $4$-dimensional and oriented and the foliation is $3$--dimensional and oriented, in which case $\chi(M)=0$, the necessary and sufficient condition for non--triviality is $M$ being almost--complex.
\end{remark}

For contrast, the case where the leaves are two dimensional is discussed in the last section. It will be shown that there is an ample class of foliations for which all foliated vector fields must have a closed orbit. 

As in the classical case, it may be possible to find classes of vector fields tangent to the foliation always possessing a periodic orbit. As an example, there is a foliated version of the Weinstein conjecture for Reeb vector fields that has been proven and disproven in several instances. See \cite{PP} for more details.

\section{Setup and applications} \label{sec:setup}

In this article, discs are assumed to be of radius $1$, unless otherwise noted. For the rest of the section $\mathcal{M}^{n+l}$ will denote a smooth compact manifold, possibly with boundary and corners. Endow $\mathcal{M}$ with a smooth $n$-dimensional foliation $\SF^n_{\mathcal M}$ and a smooth non--singular vector field $X$, tangent to $\SF_{\mathcal M}$. Homotopies of vector fields will be of particular interest and, unless stated otherwise, they will always be through smooth non--singular vector fields tangent to $\SF_{\mathcal M}$. 

By a foliated flowbox, or simply a flowbox, it is meant an embedding 
\[ \phi: [-2,2] \times \mathbb{D}^{n-1} \times \mathbb{D}^l \to \mathcal{M} \]
with image $U \subset \mathcal{M}$, a smooth submanifold with corners. In the domain of $\phi$ there are coordinates $(z;x_2,\dots, x_n;y_1,\dots, y_l)$. We require for $\phi$ to satisfy $\phi^* \SF_{\mathcal M} =  \ker(dy_1) \cap \dots \cap \ker(dy_l)$ and $\phi^* X = \partial_z$. $U^+$, $U^-$ and $U^v$ denote the components of $\partial U$ in which $X$ is outgoing, ingoing, and tangent, respectively. If $V \subset U$ is another foliated flowbox such that $V^+ \subset U^+$, $V^- \subset U^-$ and $V^v \subset \overset{\circ}{U}$ then the pair $(U,V)$ will be called nice. The following proposition is key in the construction, and the proof is standard, as in \cite[Theorem A]{Wil}.

\begin{proposition} \label{prop:main1}
Let $\mathcal{M}$, $\SF_{\mathcal M}$ and $X$ be as above. Fix $A \subset \mathcal{M}$ an open neighbourhood of $\partial \mathcal M$. Then there is a finite number of pairs $(U_i,V_i)$ satisfying:
\begin{itemize}
\item each $(U_i,V_i)$ is a nice pair of foliated flowboxes,
\item any orbit of $X$ is either fully contained in $A$ or it intersects one of the $V_i$,
\item the $U_i$ are disjoint from $\partial \mathcal M$ and disjoint from one another. 
\end{itemize}
\end{proposition}

The idea now is homotoping $X$ within the flowboxes in order to ``open up'' all closed orbits without introducing new ones. Let $N^{n-1}$ be a manifold with boundary, possibly with corners, and denote $\BN = [-2,2] \times N \times \mathbb{D}^l$, with coordinates $(z;p;y_1,\dots, y_l)$. Assume that there is an embedding $\psi: \BN \to [-2,2] \times \mathbb{D}^{n-1} \times \mathbb{D}^l$ such that:
\begin{itemize}
\item $\psi$ is the identity in the $y$ coordinates,
\item $\psi$ preserves the vertical direction, i.e. $\psi^* \partial_z = \delta\partial_z$ with $\delta$ a positive function.
\end{itemize}
If $\BN$ is endowed with a vector field $X_N$ that agrees with $\partial_z$ close to $\partial \BN$ and that has no $\partial_{y_i}$ components for all $i$, we say that the pair $(\BN, X_N)$ is a parametric plug.

Denote by $\BN^+$, $\BN^-$, and $\BN^v$ the different components of $\partial \BN$, as in the case of flowboxes. A trajectory of $X_N$ intersecting $\BN^-$ is said to be entering the plug and a trajectory intersecting $\BN^+$ is said to be exiting the plug. Since these plugs are meant to be embedded in foliated flowboxes in order to replace $X$ by $X_N$, there are a number of properties that a parametric plug must satisfy:

\begin{enumerate}
\item[i.] $X_N$ must be homotopic to $\partial_z$, relative to the boundary, and through non--vanishing vector fields with no $\partial_{y_i}$ components, $i=1,\dots,l$, \label{plug_property1}
\item[ii.] if a trajectory of $X_N$ enters and exits the plug, then it must do so at opposite points $(-2,x_0,y_0)$ and $(2,x_0,y_0)$. \label{plug_property2}
\end{enumerate}
A trajectory entering the plug and remanining there for infinite time is called trapped. The first property ensures that if the plug is used within a foliated flowbox, then the homotopy obtained is indeed through non--vanishing vector fields tangent to $\SF_{\mathcal M}$. The second property ensures that no new closed orbits are created by connecting two previously different orbits. 

\begin{proposition} \label{prop:main2}
Following with the notation of Proposition \ref{prop:main1}, suppose that there is a parametric plug {\rm $\BN$} that additionally satisfies that:
\begin{enumerate}
\item[iii.] $X_N$ has no closed orbits within {\rm $\BN$}, \label{plug_property3}
\item[iv.] the set of trajectories of $X_N$ trapped by {\rm $\BN$} contains a non--empty open set {\rm $T_N \subset  \BN^-$}. \label{plug_property4}
\end{enumerate}
Then there is a homotopy of $X$, relative to $\partial \mathcal{M}$ and through non--singular vector fields tangent to $\SF_{\mathcal M}$, to a vector field $X'$ whose closed orbits are contained in $A$. 
\end{proposition}
\begin{proof}
Given the finite cover by nice pairs $(U_i,V_i)$ as in Proposition \ref{prop:main1}, there are embeddings $\psi_i: \BN \to U_i$ satisfying that $\psi_i^* X = \partial_z$ and $V_i^- \subset \psi_i(T_N)$. This follows from the fact that $T_N$ contains a non--empty open subset by Property (iv.), which can be assumed to contain $V_i^-$ after possibly modifying $\psi_i$.

Within $\BN$, we homotope $\partial_z$ to $X_N$ as in Property (i.). Every trajectory of $X$ is either fully contained in $A$ or it intersects one $V_i^-$. In the latter case, after the homotopy the new vector field $X'$ has all its positive trajectories trapped in a plug and cannot thus be closed. Since no new closed trajectories have been introduced in the plugs by Property (iii.), the claim follows. 
\end{proof}

As soon as the existence of such a plug $\BN$ is proven for $n \geq 3$, Theorem \ref{thm:main} is an easy corollary.

\begin{proof}[Proof of Theorem \ref{thm:main} (assuming the existence of a suitable plug)]
Using the homotopy exact sequence for inclusions, the theorem is equivalent to showing that 
\[ \pi_j(\FX_{ns}(M, \SF), \FX_{no}(M, \SF)) = 0 \quad \text{ for all } j \in \mathbb{Z}. \]
Let $X_t$, $t \in \mathbb{D}^j$, be a $j$--parametric family of non--vanishing vector fields tangent to $\SF$, defining an element in $\pi_j(\FX_{ns}(M, \SF), \FX_{no}(M, \SF))$. What has to be proven now is that this family can be homotoped, leaving those $X_t$, $t \in \mathbb{S}^{j-1}$, fixed, to a family fully contained in $\FX_{no}(M, \SF)$.

Consider the manifold $\mathcal{M} = M \times \mathbb{D}^j$ with the foliation $\SF_{\mathcal M} = \coprod_{t_0 \in \mathbb{D}^j} \SF \times \{t_0\}$ of codimension $m+j$. Then $X_t$ can be regarded as a vector field $X$ in $\mathcal{M}$ tangent to $\SF_{\mathcal M}$. Since $X_t$ is an element in the relative homotopy group $\pi_j(\FX_{ns}(M, \SF), \FX_{no}(M, \SF))$, we can assume that $X$ has no closed orbits in a neighborhood $A$ of $\partial \mathcal M = M \times \mathbb{S}^{j-1}$. Then an application of Proposition \ref{prop:main2} readily implies that $X$ can be homotoped, relative to $\partial\mathcal M$ and through non--vanishing vector fields tangent to $\SF_{\mathcal M}$, to a vector field $X'$ with no closed orbits.

Equivalently, the family $X_t$ of vector fields can be homotoped, relative to the boundary of $\mathbb{D}^j$, to a family $X_t'$ fully contained in $\FX_{no}(M, \SF)$, thus proving the claim. 
\end{proof}

\section{Construction of the parametric plugs} \label{sec:plugs}

In this section we describe the parametric versions of Wilson's plug (which is needed for Theorem \ref{thm:main} if $n \geq 4$) and Kuperberg's plug (for the case $n=3)$. Note that Kuperberg's plug could be used also for the higher dimensional case, but Wilson's is easier to describe and paves the way to explain Kuperberg's.

\subsection{The Wilson Plug in dimensions 4 and higher} \label{ssec:wilson4}

Consider the manifold with boundary and corners $\BW^{n,l} = [-2,2] \times \mathbb{T}^2 \times [-2,2] \times \mathbb{D}^{n-4} \times \mathbb{D}^l$, with coordinates $(z;s,t;r;x_5,\dots,x_n;y_1,\dots,y_l)$, $s,t \in [0,\pi)$ , embedded in $\mathbb{R}^{n+l}$, $n \geq 4$ as follows:
\[ i: \BW^{n,l} \to \mathbb{R}^{n+l} \]
\[ i(z,s,t,r,x,y) = (z, \cos(s)(6+(3+r)\cos(t)), \sin(s)(6+(3+r)\cos(t)), (3+r)\sin(t),x,y). \]
Construct a vector field $X_W$ in $\BW^{n,l}$ as follows:
\[ X_W = f(z,r,x,y)(\partial_s + b\partial_t) + g(z,r,x,y)\partial_z, \]
with $b$ some irrational number and $f$, $g$ smooth functions satisfying the following constraints:
\begin{enumerate}
\item $g$ is symmetric and $f$ is antisymmetric in the $z$ coordinate,
\item $g(z,r,x,y) = 1$, $f(z,r,x,y) = 0$ close to the boundary of $\BW^{n,l}$,
\item $g(z,r,x,y) \geq 0$ everywhere and $g(z,r,x,y) = 0$ only in $\{|z| = 1, |r| \leq 1, |x| \leq 1/2, |y| \leq 1/2\}$,
\item $f(z,r,x,y) = 1$ in $\{z \in [-3/2,-1/2], |r| \leq 1, |x| \leq 1/2, |y| \leq 1/2 \}$. 
\end{enumerate}

This is the usual construction for Wilson's plug, but we have explicitely split the additional coordinates into $(x_i)_{i=5,\dots,n}$ and $(y_j)_{j=1,\dots,l}$, so that the $y$ coordinates denote the parameter space. Write $\BW^{n,l}_{y_0}$ for the $n$--dimensional plug one obtains for $y=y_0$ fixed. 

\begin{proposition} $($Wilson \cite{Wil}$)$
Wilson's plug satisfies all $4$ properties required for Proposition \ref{prop:main2} to hold.
\end{proposition}
\begin{proof}
Property (i.) follows by interpolating linearly between $g$ and the constant function $1$ and then between $f$ and the constant function $0$. The symmetry of $g$ and the antisymmetry of $f$ imply Property (ii.). The only possible closed orbits within $\BW^{n,l}$ would lie in the zero set of $g$, and by construction the flow in the zero set consists of invariant tori in which the vector field has irrational slope, so Property (iii.) follows. Finally, the orbits touching $\{z= -2, |r| \leq 1, |x| \leq 1/2, |y| \leq 1/2\}$ are trapped, proving Property (iv.). 
\end{proof}

\subsection{The Wilson plug in dimension 3} \label{ssec:wilson3}

It is clear from the construction above that Wilson's method cannot be used in dimension $3$. However, a 3--dimensional version can be constructed. This object will be used later on when defining Kuperberg's plug. The treatment here follows very closely the one in \cite{HR}, where everything is described in more detail. \

Consider the manifold $\BW = [-2,2] \times \mathbb{S}^1 \times [1,3]$, with coordinates $(z, \theta, r)$, embedded in $\mathbb{R}^3$ cylindrically  in the obvious fashion. Define a vector field $X_S$ in $\BW$ as follows:
\[ X_W = f(z,r)\partial_\theta + g(z,r)\partial_z, \]
with the functions $f$ and $g$ satisfying:
\begin{itemize}
\item $f$ is antisymmetric and $g$ is symmetric in the $z$ coordinate,
\item $f$ is $0$ and $g$ is $1$ near the boundary of $\BW$,
\item $g(z,r) \geq 0$ and $g(z,r) = 0$ only in $B = \{|z| = 1, r=2\}$,
\item $f(z,r) \geq 0$ in $\{|z| > 0\}$ and $f(z,r) = 1$ in $\{1/4 \leq z \leq 7/4, 5/4 \leq r \leq 11/4 \}$.
\end{itemize}

This version of Wilson's plug satisfies Properties (i.) and (ii.), as is easily verified. Further, it contains a pair of closed orbits, namely, $\{|z| = 1, r=2\}$ and a closed set of orbits that get trapped, those touching $\{z= -2, r = 2\}$. Observe that the flow of $X_W$ is tangent to the cylinders with $r = r_0$ fixed.

\subsection{The Kuperberg Plug} \label{ssec:kuperberg}

The plug $\BW$ described above is the basis for Kuperberg's plug. $\BW$ can be inmmersed in $\mathbb{R}^3$ in a non--standard manner, by inserting it into itself in order to eliminate the two periodic orbits $\gamma_i$, $i=1,2$, that it has. See \cite{Kup} for the original article and \cite{HR} for a very detailed account of the construction. 

The key objects are as follows. There are two disjoint areas $L_i \subset \mathbb{S}^1 \times [1,3]$, $i=1,2$, and corresponding cylinders $D_i = [-2,2] \times L_i$ that are going to be reinserted into the plug. Let $\mathcal{D}_i$, $i=1,2$, be two disjoint flowboxes satisfying:
\begin{itemize}
\item each $\mathcal{D}_i$ contains an interval $\{((-1)^i,\theta,2)$, $\theta^{-}_i \leq \theta \leq \theta^{+}_i\}$, of one of the two closed orbits $\gamma_i$,
\item each $\mathcal{D}_i$ is diffeomorphic, as a manifold with boundary and corners, to $D_i$ by a map $\sigma_i: D_i \to \mathcal{D}_i$ satisfying $\sigma_i^* X_W = \partial_z$. Denote $\SL_i^{\pm} = \sigma_i(\{\pm 2\} \times L_i )$, for $i=1,2$. 
\end{itemize}
Both $D_i$ and $\mathcal{D}_i$ have part of their boundaries contained in $\partial \BW$. These properties imply that the identification $\sigma_i$ can be realised by an immersion with self--intersections of $\BW$ into $\mathbb{R}^3$ that allows the flow $X_W$ in $\mathcal{D}_i$ to be replaced by $(\sigma_i)_* X_W$. See Fig.~\ref{fig:kuperberg} for a picture of all these elements. 

Denote $\{(z,\theta_i,r_0)$, $z \in [-2,2]\} \subset D_i$ the vertical interval that is the preimage of $\gamma_i \cap \mathcal{D}_i$ under $\sigma_i$. Then we further require for the following property to hold:

\begin{itemize}
\item \textsl{Radius inequality:} ``for all $(z,\theta,r) \in D_i$, with image $\sigma_i(z,\theta,r) = (z',\theta', r') \in \mathcal{D}_i$, it holds that $r' < r$ except for the points $(z,\theta_i,2)$, where it is actually an equality."
\end{itemize}

The quotient manifold constructed by identifying in $\BW$ the cylinders $D_i$ and $\mathcal{D}_i$ using $\sigma_i$ will be denoted $\BK$, see Fig.~\ref{fig:kuperberg}. The quotient vector field obtained out of $X_W$ by replacing it with $(\sigma_i)_* X_W$ in $\mathcal{D}_i$ will be denoted $X_K$.

\begin{figure}[ht]
\centering
\includegraphics[scale=0.5]{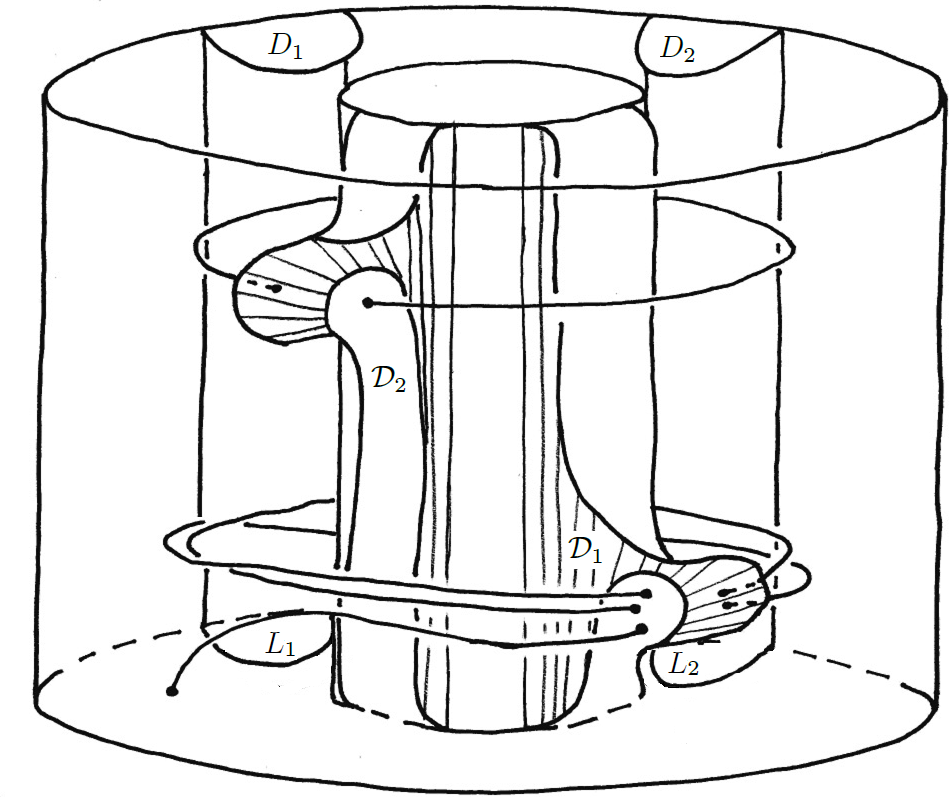}
\caption{The Kuperberg manifold seen as a quotient of the Wilson cylinder. $D_i$ is identified with $\mathcal{D}_i$, $i=1,2$. This figure originally appears in \cite{Ghy} and \cite{HR}.}\label{fig:kuperberg}
\end{figure}

The following theorem of Matsumoto shows that Property (iv.) of plugs is satisfied by Kuperberg's plug.

\begin{theorem} $($\cite{Mat}$)$ \label{thm:Matsumoto}
There is $\delta > 0$ such that every orbit entering the Kuperberg plug at $\{\pm 2\} \times \mathbb{S}^1 \times (2-\delta,2)$ is trapped inside.
\end{theorem}

The following Lemma will be useful in the next subsection.

\begin{lemma} \label{lem:paramWilson}
There is a homotopy in {\em$\BW$} of non--singular vector fields $X_W^t$, $t \in [0,2]$, with $X_W^0 = X_W$ and $X_W^1 = \partial_z$, such that:
\begin{itemize}
\item $X_W^t$ agrees with $\partial_z$ in $D_i$ for $t \in [1,2]$,
\item $X_W^t$ agrees with $X_W$ in $\mathcal{D}_i$ for $t \in [0,1]$,
\item $X_W^t$ defines a plug with no closed nor trapped orbits for $t > 0$.
\end{itemize}
\end{lemma}
\begin{proof}
Let $f$ and $g$ be the defining functions for $X_W = f(z,r)\partial_\theta + g(z,r) \partial_z$. Fix disjoint open subintervals of the circle $I_i, \mathcal{I}_i \subset \mathbb{S}^1$, $i=1,2$, such that $D_i \subset [-2,2] \times I_i \times [1,3]$ and $\mathcal{D}_i \subset [-2,2] \times \mathcal{I}_i \times [1,3]$. Fix slightly larger intervals $I_i', \mathcal{I}_i'$, still disjoint, such that $I_i \subset I_i'$ and $\mathcal{I}_i \subset \mathcal{I}_i'$. Construct bump functions 
\[ \alpha, \beta: \mathbb{S}^1 \to [0,1] \]
\[ \alpha(p) = 1, p \in I_i; \quad \alpha(p) = 0, p \notin I_i'; \quad i=1,2, \]
\[ \beta(p) = 1, p \in \mathcal{I}_i; \quad \beta(p) = 0, p \notin \mathcal{I}_i'; \quad i=1,2. \]

Let $\phi: [0,2] \to [0,1]$ be a smooth function that is increasing in $[0,1]$ and satisfies $\phi(0) = 0$ and $\phi(t) = 1$ for $t \in [1,2]$. Similarly, let $\psi: [0,2] \to [0,1]$ be a smooth function such that $\psi(t) = 0$ for $t \in [0,1]$ and $\psi(2) = 1$. Now define:
\[ f_t(z,\theta,r) = f(z,r)(1 - \phi(t)\alpha(\theta) - \psi(t)(1-\alpha(\theta))) \]
\[ g_t (z,\theta,r) = g(z,r) + (1-g(z,r))(\phi(t)\alpha(\theta) + \psi(t)(1-\alpha(\theta))) \]
\[ X_W^t = f_t\partial_\theta +  g_t\partial_z. \]

It is immediate that $X_W^t$ is non--singular and that the first two claims hold. For the last one, observe that $g_t > C_t > 0$ in $I_i$ for $t>0$, with $C_t$ some positive constant.
\end{proof} 

See Fig.~\ref{fig:homotopy} for a pictorial representation of this construction. 

\begin{figure}[ht]
\centering
\includegraphics[scale=0.33]{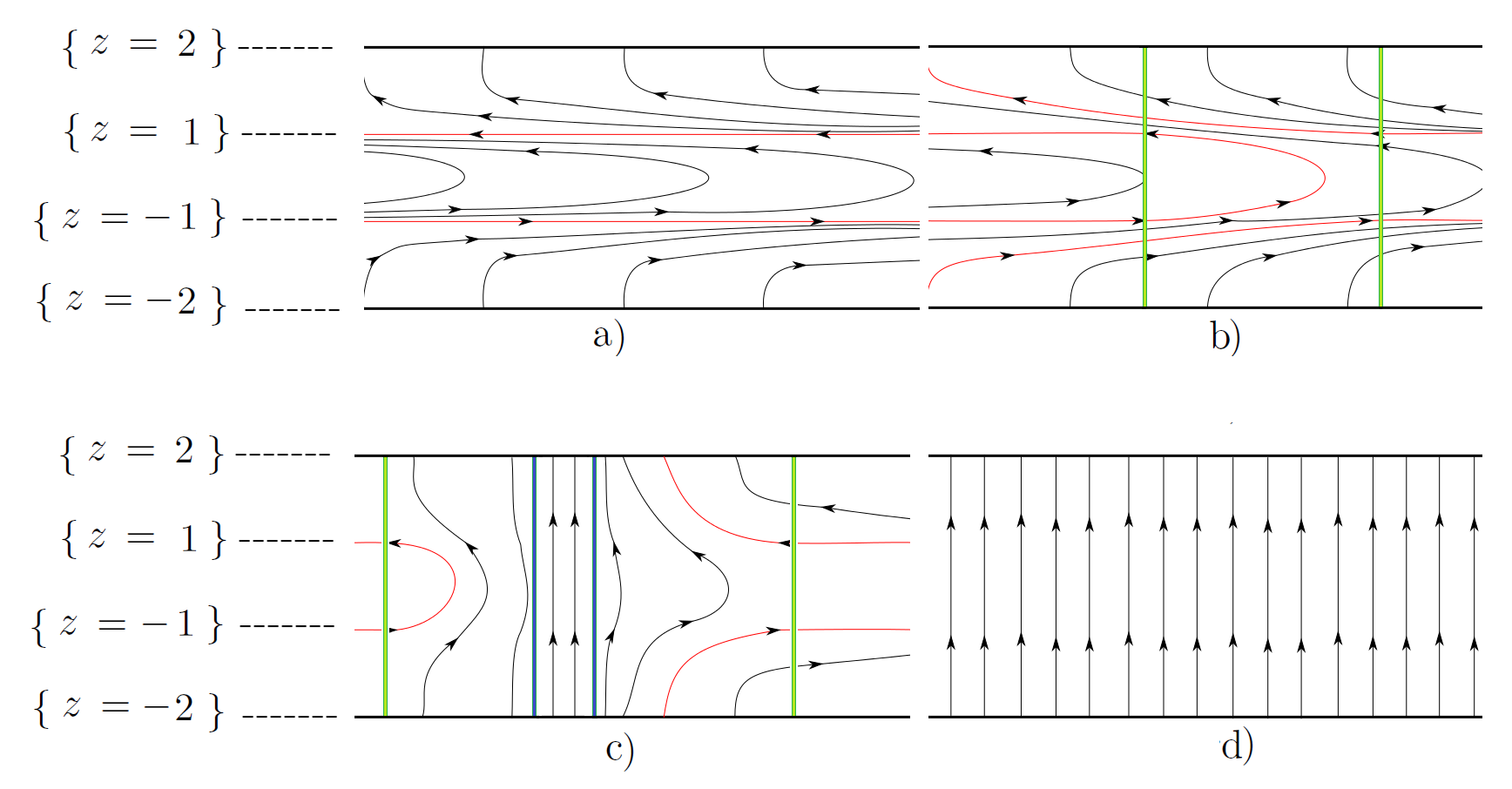}
\caption{The flow of $X_W^t$ at $\{r=2\}$. Image $a)$ corresponds to $t=0$, $b)$ to $t=1/2$, $c)$ to $t=1$, and $d)$ to $t=2$. The areas enclosed by the green lines correspond to the intervals $I_i'$, and the area enclosed by the blue lines to $I_i$. The red flow line corresponds to the orbit(s) that is(are) tangent to the curves $\{|z|=1\}$ outside of $I_i'$.}\label{fig:homotopy}
\end{figure}

\subsubsection{The parametric Kuperberg plug}

The \textsl{radius inequality} is the key to showing that the Kuperberg plug traps a non--empty open set of orbits and that it contains no closed orbits. Similarly, consider the following property:
\begin{itemize}
\item The \textsl{strict radius inequality} holds for a diffeomorphism $\phi_i: D_i \to \mathcal{D}_i$ if $r' < r$ for every $(z,\theta,r) \in D_i$ with $\phi_i(z,\theta,r) = (z',\theta', r')$.
\end{itemize}
In the process of interpolating to a trivial plug, we will need for the intermediate plugs to satisfy this strict radius inequality, since it will guarantee that all orbits enter and exit the plug. 

A family of diffeomorphisms 
\[ \sigma_i^t: D_i \to \mathcal{D}_i, \quad t \in [0,2], i=1,2; \quad \text{satisfying} \]
\[ \sigma_i^0 = \sigma_i; \quad (\sigma_i^t)^* X_W = \partial_z \]
and satisfying the strict radius inequality for $t > 0$ can be constructed easily. The diffeomorphisms $\sigma_i$ can be precomposed with diffeomorphisms of $D_i$ that preserve the $z$ component, that restrict to the identity in $[-2,2] \times (\partial L_i)$ and that, away from there, take points to points with smaller radius. This produces diffeomorphisms $\sigma_i^t$ that are $C^\infty$--close to $\sigma_i$. The quotients of $\BW$ induced by the gluings $\sigma_i^t$ are all diffeomorphic to the Kuperberg manifold $\BK$ and it is possible to fix a smooth $t$--parametric family of identifications with $\BK$, which we henceforth assume. 

Recall the explicit homotopy $X_W^t$ constructed in Lemma \ref{lem:paramWilson}. We define a family of vector fields in $\BW$ as follows:
\begin{itemize}
\item $Y_W^t = X_W^t$ in $(\BW \setminus (D_1 \cup D_2 \cup \mathcal{D}_1 \cup \mathcal{D}_2))$,
\item $Y_W^t = X_W^t$ in $D_i$ for $t \in [0,1]$,
\item $Y_W^t = (\sigma_i^t)_* X_W^{t}$ in $\mathcal{D}_i$ for $t \in [0,1]$,
\item $Y_W^t = X_W^t$ in $\mathcal{D}_i$ for $t \in [1,2]$,
\item $Y_W^t = (\sigma_i^t)^*X_W^t$ in $D_i$ for $t \in [1,2]$.
\end{itemize}	
Note that this vector field does not define a plug in $\BW$, since it is not vertical close to the boundary in $D_i$ for $t \in [1,2]$. However, it does descend to the quotient $\BK$ and automatically induces a family of plugs $(\BK, X_K^t)$, $t \in [0,2]$ interpolating from $X_K=X_K^0$ to $\partial_z=X_K^2$. 
 
\begin{lemma} \label{lem:triviality}
$(\BK, X_K^t$) has no closed orbits. Further, for $t>0$, all orbits enter and exit the plug at opposing points. 
\end{lemma}
\begin{proof}
Since $X_K^0$ is Kuperberg's plug, it has no closed orbits. Let us set up some notation for the case $t > 0$. There are smooth bijective projections 
\[ \tau: \BW \setminus (\mathcal{D}_1 \cup \mathcal{D}_2) \to \BK, \]
\[ \tau': \BW \setminus (D_1 \cup D_2) \to \BK. \]
The discontinuous radius function $\rho: \BK \to [1,3]$ at a point $p$ is defined to be the radius of $\tau^{-1}(p)$. Similarly, $r(p)$ will be the radius of the preimage by $\tau'$. Compactness of $D_i$ and the \textsl{strict radius inequality}, imply that there is a lower bound 
\begin{equation} \label{eq:radiusBound}
\rho - r \geq \epsilon > 0 
\end{equation}
in the points where they disagree. 

Fix a point $p \in \BK$. Let $\Phi_s(p)$ be the flow of $X_K^t$ at time $s$ valued at $p$. We can denote, for $i=1,2$: 
\[ E_i(s_0) = \{ \Phi_s(p) \cap \tau'(\SL_i^-); s \in (0,s_0)\} \]
\[ S_i(s_0) = \{ \Phi_s(p) \cap \tau'(\SL_i^+); s \in (0,s_0)\} \]
the sets of points where the forward orbit of $p$ enters and exits, respectively, the self--insertion of the plug. Define $E_i = \cup_{s\geq 0} E_i(s)$ and $S_i = \cup_{s \geq 0} S_i(s)$. Define the level function associated to $p$ as follows:
\[ \nu_p(s) = (\#E_1(s) + \#E_2(s)) - (\#S_1(s) + \#S_2(s)), \quad s \geq 0.  \]
Consider the collection of points $E_1 \cup E_2 \cup S_1 \cup S_2$ and regard it as an ordered list $L = \{x_j = \Phi^t_{s_j}(p)\}$ in terms of increasing $s_j$, so the points appear in $L$ as the forward orbit intersects the sets $\mathcal{L}^{\pm}_i$.

Consider two points $x = x_j$ and $y = x_k$, with $x$ an entry point. Take the list $\{x_i\}_{i \in \{j,\dots,k\}} \subset L$ of points lying inbetween. If $\nu_p(x) = \nu_p(y) \geq \nu_p(x_i)$, $i \in \{j+1,\dots, k-1\}$, then we claim that $\tau^{-1}x$ and $\tau^{-1}y$ lie in the same Wilson orbit. We proceed by induction on the size of $\{j,\dots,k\}$. The base case $k=j+1$ is immediate by construction. For a list of length $2n$, $n > 1$, there are two cases. 

If $x_{j+1}$ is an entry point, then the claim holds for the list of length $2n-2$ defined by $x_{j+1}$ and $x_{k-1}$. This implies that $\tau^{-1}x_{j+1}$ and $\tau^{-1}x_{k-1}$ are joined by a Wilson orbit going from the bottom boundary of $\BW$ to the top one, so in particular, $\tau^{-1}x_{j+1}$ and $\tau^{-1}x_{k-1}$ must be at opposing points of the boundary of $\BW$. Write $\gamma$ for the vertical segment joining them. Then the Wilson segment joining $x$ and $y$ is given by the concatenation of the segment joining $x$ and $x_{j+1}$, $\tau \circ \gamma$ and the segment joining $x_{k-1}$ and $y$. 

If $x_{j+1}$ is an exit point, and $j+2 = k-1$, then $x_{k-1}$ is an entry point and the Kuperberg orbit inbetween $x$ and $y$ agrees with the Wilson orbit, proving the claim. If $j+2 < k-1$, $x_{j+2}$ is an entry point and the induction step can be applied to the list $\{x_i\}_{i \in \{j+2,\dots,k-2\}}$. Again, the Kuperberg orbit between $x$ and $x_{j+2}$ agrees with the Wilson orbit and the same is true for $x_{k-2}$ and $y$. Pushing forward by $\tau$ the vertical segment joining $x_{j+2}$ and $x_{k-2}$ provided by the induction step provides the missing segment in the Wilson orbit joining the previous two. 

Observe that this means that $\rho(x) = \rho(y)$ and $r(x) = r(y)$, since the radius remains constant in Wilson orbits. This fact taken together with Equation \ref{eq:radiusBound} implies that the elements in the list $N = \{\nu_p(s_j)\}$ have an upper bound, because $\rho$ cannot be arbitrarily large. If $L$ is infinite, then there is a minimum number $k$ that gets repeated infinitely many times in $N$.

By minimality of $k$, we can choose $x$ and $y$ with $\nu_p(x) = \nu_p(y) = k$, all points inbetween with greater $\nu_p$, and defining an arbitrarily long list of points. But this is a contradiction, since this means that we can find Wilson segments that intersect $\SL^\pm_i$ arbitrarily many times. Therefore, $L$ must be finite and every orbit eventually escapes the plug. A similar analysis for negative time shows that it must enter the plug too. The level analysis above shows that it must do so at opposing points.
\end{proof}

Construct a smooth non--decreasing function $\eta: [0,1] \to [0,2]$, satisfying:
\begin{itemize}
\item $\eta$ is identically $0$ in $[0,1/2]$,
\item $\eta > 0$ in $(1/2,1]$,
\item $\eta$ is identically $2$ close to $1$.
\end{itemize}
Define a family of functions $\eta = (1-s)\eta + 2s$, $s \in [0,1]$. Let $\mathbb{D}^l$ be the disk with coordinates $(y_1,\dots,y_l)$. A $1$--parametric family of foliated vector fields $\mathcal{X}_K^s$ in $\BK \times \mathbb{D}^l$ can be defined by 
\[ (\mathcal{X}^s_K)|_{\{y=y_0\}} = X_K^{\eta_s(|y_0|)}. \]

\begin{proposition}
$(\BK, \mathcal{X}_K^0)$ satisfies all $4$ properties required for Proposition \ref{prop:main2} to hold.
\end{proposition}
\begin{proof}
$\mathcal{X}_K^s$ is the necessary homotopy between $\mathcal{X}_K^0$ and $\mathcal{X}_K^1 = \partial_z$. That this homotopy is through non--vanishing foliated vector fields follows from the fact that the $X_K^t$ were non--vanishing. Property (i.) holds. 

A theorem of Matsumoto \cite{Mat} states that Kuperberg's plug traps a non--empty open set of orbits $T_\BK$. Since $(\mathcal{X}_K^0)_{\{y=y_0\}}$ agrees with the vector field in Kuperberg's plug for $y_0 \in \mathbb{D}^l_{1/2}$, it is immediate that $(\BK, \mathcal{X}_K^0)$ traps the open set $T_\BK \times \mathbb{D}^l_{1/2}$. Property (iv.) follows.

For $|y_0| > 1/2$ it holds that $\rho(|y_0|) > 0$. Hence, applying Lemma \ref{lem:triviality} to the flow $(\mathcal{X}^0_K)|_{\{y=y_0\}} = X_K^{\rho(|y_0|)}$ shows that $\mathcal{X}^0_K$ has no closed orbits in $|y_0| > 1/2$ and all orbits there go through the plug entering and exiting at opposing points. For $|y_0| \leq 1/2$, $(\BK, (\mathcal{X}^0_K)|_{\{y=y_0\}})$ is the Kuperberg plug. This proves Properties (ii.) and (iii.).
\end{proof}

\section{Foliations with leaves of dimension 2} \label{ssec:dim2}

In this section $M^3$ will denote a connected orientable compact smooth $3$-manifold, possibly with boundary. It will be endowed with a $2$-dimensional foliation $\SF^2$, which is assumed to be orientable and tangent to the boundary of $M$. Further, let $X$ be a non--singular vector field tangent to $\SF$.

\begin{lemma} \label{lem:standardForm}
Let $(\mathbb{T}^2, \SF_T)$ be a smooth foliation by lines in the torus. If $\SF_T$ has no Reeb components, then it is equivalent, up to conjugation by a homeomorphism of $\mathbb{T}^2$, to the foliation induced by the suspension of a diffeomorphism of the circle. If $\SF_T$ has no closed orbits then the diffeomorphism of the circle is an irrational rotation.
\end{lemma}

This is a well known fact. A proof can be found in \cite{HeHi}. The following proposition establishes the existence of at least two periodic orbits for any vector field tangent to the standard Reeb component. The corollary after the proposition is an immediate consequence of Novikov's compact leaf theorem.

\begin{proposition} \label{prop:standardReeb}
Let $(M^3, \SF^2)$ be a standard Reeb component. Then $X$ induces a Reeb component on its boundary torus. In particular, $X$ has at least $2$ closed orbits.
\end{proposition}
\vspace{-1em}
\begin{proof}
Denote by $\mathcal{X}$ the oriented foliation by lines induced by $X$ on the boundary torus $T$ of the Reeb component. Assume that $\mathcal{X}$ has a Reeb component in $T$. Since this foliation is orientable, the Reeb component cannot have as boundary a single leaf $\mathbb{S}^1$, so the vector field $X$ must have at least $2$ closed orbits. Let us now assume that $\mathcal{X}$ does not have a Reeb component.

Parametrise $M = \mathbb{D}^2 \times \mathbb{S}^1$ explicitely with coordinates $(r,\theta,t)$, $|r| \leq 1$. Consider the one sided neighbourhood $\phi: (0,1] \times \mathbb{T}^2 \to M$, $\phi(r,\theta, t) \to (\frac{1+r}{2},\theta,t)$ of the boundary torus $T$. Any curve representing the homology class $m \in H_1(T;\mathbb Z)$ that vanishes by inclusion into $M$ is called a meridian.

 Using Lemma \ref{lem:standardForm} in the $\mathbb{T}^2$ coordinates yields a new (maybe topological) embedding $\psi: (0,1] \times \mathbb{T}^2 \to M$ such that $\psi^* \mathcal{X}$ is a suspension of a diffeomorphism of the circle in the torus $\{1\} \times \mathbb{T}^2$. 

Suppose that $\psi^* \mathcal{X}$ corresponds to the irrational rotation, then any curve with rational slope makes a constant angle with $\psi^* \mathcal{X}$. Note that, in particular, the homology class $(\psi|_{\{1\} \times \mathbb{T}^2})^* m$ of the meridian under this new parametrisation can be represented by some smooth curve $\gamma$ with rational slope. Accordingly, $\psi^* \mathcal{X}$ and the tangent vector $\overset{.}{\gamma}$ define, at each point in the image of $\gamma$, a positively oriented basis.

Suppose instead that $\psi^* \mathcal{X}$ corresponds to a suspension of a diffeomorphism of $\mathbb{S}^1$ with fixed points. The meridian class $(\psi|_{\{1\} \times \mathbb{T}^2})^* m$ can be represented by a smooth curve $\gamma: \mathbb{S}^1 \to \{1\} \times \mathbb{T}^2$. Denoting this class by $(a,b)$, where the first component stands for the suspension direction, the curve $\gamma$ can be set to agree with a compact leaf of $\mathcal X$ for almost $a$ turns and then to turn $b$ times transversely. Accordingly, the foliation $\psi^* \mathcal{X}$ and the tangent vector $\overset{.}{\gamma}$ are, at each point in the image of $\gamma$, either colinear or define a positively oriented basis.

Summarizing, if the foliation $\psi^* \mathcal{X}$ does not have a Reeb component, it admits a smooth curve $\gamma: \mathbb{S}^1 \to \{1\} \times \mathbb{T}^2$ representing the meridian class $(\psi|_{\{1\} \times \mathbb{T}^2})^* m$, such that $\psi^* \mathcal{X}$ and $\overset{.}{\gamma}$ are either colinear or define a positively oriented basis at every point. The degree of $\psi^* \mathcal{X}$ restricted to the image of $\gamma$ is therefore $0$. Since the degree is invariant by homeomorphism, we conclude that $\mathcal X$ has degree $0$ on the image of the curve $\psi \circ \gamma$.

Now every leaf inside the Reeb component has a family of circles that asymptotically approach the image of $\psi \circ \gamma$. The previous discussion implies that $X$ restricted to any given $\mathbb{R}^2$ leaf in the Reeb component is a non--singular vector field that restricted to some circle has degree $1$ (with respect to the standard basis of $\mathbb{R}^2$). Using Poincar\'e-Hopf index theorem we get a contradiction, thus implying that $X$ has a Reeb component in the boundary torus $T$, as we desired to prove. 
\end{proof}

\begin{corollary}
Any non--singular vector field tangent to a codimension one foliation of $\mathbb S^3$ has at least $2$ closed orbits.
\end{corollary}

Proposition~\ref{prop:standardReeb} can be proved in more generality. Following~\cite{IY76} and~\cite{Sch11} we introduce the following definition.

\begin{definition}
A foliation $(M, \SF)$ is called a generalised Reeb component if $M$ is connected, $\partial M$ is a union of leaves of $\SF$, no couple of points on $\partial M$ can be joined by a curve transverse to the foliation, and all the leaves in $\SF|_{\overset{\circ}{M}}$ are proper and without holonomy.
\end{definition}

In particular, this means that $\partial M$ is a union of tori. The following lemma, which is a straightforward consequence of~\cite[Corollary 2]{Ima76} and~\cite[Theorem 1]{Ni}, states that the behaviour near the boundary components is just like the one found in a standard Reeb component:

\begin{lemma} \label{lem:reebLocalModel}
Let $(M, \SF)$ be a generalised Reeb component and let $T \subset \partial M$ be one of the boundary components. Then the one sided holonomy along $T$ is an infinite cyclic group. In particular, there is a basis $(\alpha, \beta)$ for $H_1(T)$ such that the holonomy along $\alpha$ is contracting and the holonomy along $\beta$ is the identity. 
\end{lemma}

We shall see in Theorem \ref{thm:generalisedReeb} below that in most generalised Reeb components $X$ must carry closed orbits. First we characterise the exceptions. Consider the annulus $\mathbb{S}^1 \times [0,1]$ and denote by $\SF_R$ the $1$-dimensional Reeb foliation on the annulus. We will abuse notation and still denote by $\SF_R$ its lift as a codimension one foliation to $\mathbb{T}^2\times [0,1]$.

\begin{lemma}
Let $(M^3, \SF^2)$ be a generalised Reeb component. Suppose one of the leaves $F$ is a cylinder. Then $(M,\SF)$ is homeomorphic to $(\mathbb{T}^2 \times [0,1],\SF_R)$.
\end{lemma}
\vspace{-1em}
\begin{proof}
By~\cite{Ima76} it follows that $(\overset{\circ}{M},\SF|_{\overset{\circ}{M}})$ is a fibration over $\mathbb S^1$ whose leaves are diffeomorphic to cylinders. Since $M$ is orientable, the fibration $\pi: \overset{\circ}{M} \to \mathbb{S}^1$ is trivial.

Let $\phi_i: (0,1] \times \mathbb{T}^2 \to M$, $i=1,2$, be one--sided charts of the $2$ boundary components $\phi_i(\{1\} \times \mathbb{T}^2)$, with coordinates $(r,s,\theta)$. By Lemma \ref{lem:reebLocalModel}, it can be assumed that the holonomy is the identity in the $s$--direction and contracting in the $\theta$--direction. Then these local models can be assumed to agree with that of the standard Reeb component.

Since the leaves are proper, there are numbers $r_1, r_2$, such that the tori $S_i = \{r=r_i\} \subset \im(\phi_i)$ intersected with each leaf bound a compact cylinder. Then the $\phi_i$ can be reparametrised in the $\theta$--direction so that $\pi \circ \phi_i^{-1}(r_i,s,\theta) = \pm \theta$. The sign depends on whether the coorientation of $\SF$ agrees with the direction in which the holonomy is contracting. Denote by $B \subset \overset{\circ}{M}$ the manifold bounded by the tori $S_i$. Since $\pi: B \to \mathbb{S}^1$ is a submersion that is a fibration over each $S_i$, the Ehresmann fibration theorem implies that $B$ is a trivial $\mathbb{S}^1 \times (-1,1)$ bundle over $\mathbb{S}^1$. 

The boundary torus $S_i$ is endowed with two trivialisations, one coming from $B$ and the other from $\phi_i$. They might disagree by a number of Dehn twists in the $s$--direction. Denote their composition by $\tau: \mathbb{T}^2 \to \mathbb{T}^2$. Since the foliation structure in the chart $\phi_i$ is invariant under the action of $\tau$ on the $(s,\theta)$ coordinates, $\psi_i = \phi_i \circ \tau^{-1}$ is a new chart structure that makes the two trivialisations of $S_i$ agree. Therefore, the trivialisation from $B$ glues with the charts $\psi_i$ to yield $\mathbb{T}^2 \times [0,1]$ as a manifold. Further, if $\pi \circ \psi_1^{-1}(r_1,s,\theta) = \pi \circ \psi_2^{-1}(r_2,s,\theta)$, then $\SF$ is isomorphic to $\SF_R$. Otherwise, that is if the orientations of the boundary components are reversed, $(M^3, \SF^2)$ has a transverse path connecting two points of the boundary and is not a generalised Reeb component.
\end{proof}

Now the main result is immediate:

\begin{theorem} \label{thm:generalisedReeb}
Let $(M^3, \SF^2)$ be a generalised Reeb component. If $M$ is not homeomorphic to $\mathbb{T}^2 \times [0,1]$, then any vector field $X$ tangent to $\SF$ has at least $2$ closed orbits.
\end{theorem}
\begin{proof}
Since $M$ is not homeomorphic to $\mathbb{T}^2 \times [0,1]$, none of the non--compact leaves of $\SF$ are cylinders. In particular, they must have non--zero euler characteristic. Assume that $X$, when restricted to all boundary components of $M$, induces no Reeb component. Applying Lemma \ref{lem:reebLocalModel} and proceeding as in Proposition \ref{prop:standardReeb} shows that, given some non--compact leaf $F$, there is a finite collection of closed curves $\gamma_i \subset F$ satisfying:
\begin{itemize}
\item $F \setminus \{\gamma_i\}$ is comprised of a compact component $G$ that is a deformation retract of $F$ and a collection of non--compact half--cylinders,
\item $X$ is either tangent or defines a positively oriented basis at each point of $\gamma_i$ (endowed with appropriate orientations). 
\end{itemize}
These properties again yield a contradiction using the Poincar\'e--Hopf index theorem.
\end{proof}

\section*{Acknowledgments}
The authors are thankful to Gilbert Hector for clarifying some concepts and explaining some results regarding Reeb components. We are also greatly indebted to Ana Rechtman and Steve Hurder for letting us use Figure~\ref{fig:kuperberg}. D.P.-S. is supported by the ERC starting grant~335079. A.P. and F.P. are supported by the Spanish National Research Project MTM2010--17389. This work is supported in part by the ICMAT Severo Ochoa grant SEV-2011-0087, and it is part of the last two authors' activities within CAST, a Research Network Program of the European Science Foundation.

\end{document}